\documentclass[12pt,twoside]{amsart}
\usepackage{amsmath}
\usepackage{enumerate}
\usepackage{amsthm}
\usepackage{amsfonts}
\usepackage{amssymb}
\usepackage{latexsym}
\usepackage{mathrsfs}
\usepackage{amsmath}
\usepackage{amsthm}
\usepackage{amsfonts}
\usepackage{amssymb}
\usepackage{latexsym}
\usepackage{geometry}
\usepackage{dsfont}
\usepackage[dvips]{graphicx}
\usepackage{color}
\usepackage[all]{xy}

\date{}
\allowdisplaybreaks[4] \footskip=15pt
\renewcommand{\uppercasenonmath}[1]{}

\usepackage{graphicx,amssymb}
\usepackage[all]{xy}
\usepackage{amsmath}

\allowdisplaybreaks
\usepackage{amsthm}
\usepackage{color}

\theoremstyle{plain}
\newtheorem{theorem}{Theorem}[section]
\newtheorem{proposition}[theorem]{Proposition}
\newtheorem{lemma}[theorem]{Lemma}
\newtheorem{corollary}[theorem]{Corollary}
\theoremstyle{definition}
\newtheorem{example}[theorem]{Example}
\newtheorem{definition}[theorem]{Definition}
\newtheorem{question}[theorem]{Open Question}
\theoremstyle{definition}
\newtheorem*{acknowledgement}{Acknowledgement}

\theoremstyle{remark}
\newtheorem{remark}[theorem]{Remark}

\newcommand{\pf}{\noindent\begin {proof}}
\newcommand{\epf}{\end{proof}}

\newcommand{\Ker}{\mbox{\rm Ker}}
\newcommand{\Ext}{\mbox{\rm Ext}}
\newcommand{\Hom}{\mbox{\rm Hom}}
\newcommand{\Tor}{\mbox{\rm Tor}}

\newcommand{\Prufer}{Pr\"{u}fer}

\def\p{{\frak p}}

\def\fd{{\rm fd}}
\def\pd{{\rm pd}}

\def\cwd{{\rm w.gl.dim}}
\def\cgd{{\rm gl.dim}}

\def\Hom{{\rm Hom}}
\def\Ext{{\rm Ext}}
\def\Tor{{\rm Tor}}

\def\fkm{{\frak m}}

\def\ker{{\rm ker}}
\def\Ker{{\rm Ker}}

\def\Coker{{\rm Coker}}

\def\Nil{{\rm Nil}}
\def\NN{{\rm NN}}

\def\NP{{\rm NP}}
\def\gl{{\rm gl.dim}}
\def\Z{{\rm Z}}

\def\T{{\rm T}}

\def\ZN{{\rm ZN}}

\def\Krull{{\rm Krull}}
\def\Dedekind{{\rm Dedekind}}

\def\Prufer{{\rm Pr\"{u}fer}}

\begin{document}
\begin{center}
{\large  \bf Strongly $\phi$-flat modules, strongly nonnil-injective modules and their homological dimensions}

\vspace{0.5cm}   Xiaolei Zhang$^{a}$, Shiqi Xing$^b$, Wei Qi$^{a}$\\

{\footnotesize a.\  School of Mathematics and Statistics, Shandong University of Technology, Zibo 255000, China\\

b.\  College of Applied Mathematics, Chengdu University of Information Technology, Chengdu 610225, China\\


Corresponding Author: Xiaolei Zhang, E-mail: zxlrghj@163.com\\}
\end{center}

\bigskip
\centerline { \bf  Abstract}
\bigskip
\leftskip10truemm \rightskip10truemm \noindent

In this paper, we first introduce and study  the notions of strongly $\phi$-flat modules and strongly nonnil-injective modules. And then, we  investigate the  homological dimensions of modules and rings in terms of these two notions. Finally we give some new homological characterizations of $\phi$-Dedekind rings and  $\phi$-\Prufer\ rings.
\vbox to 0.3cm{}

\leftskip0truemm \rightskip0truemm
\bigskip

Throughout this paper, all rings are commutative with identity and all modules are unitary.
First, we recall some notions on $\phi$-rings, which are good generalizations of integral domains,   originated from \cite{A97}. A ring $R$ is called an \emph{$\NP$-ring} if the nilpotent radical $\Nil(R)$ is a prime ideal; and a \emph{$\ZN$-ring} if $\Z(R)=\Nil(R)$ where $\Z(R)$ is the set of all zero-divisors of $R$. A prime ideal $\p$ of $R$ is called \emph{divided prime} if $\p\subsetneq (x)$, for every $x\in R-\p$. A ring $R$ is a \emph{$\phi$-ring} if $\Nil(R)$ is a divided prime ideal of $R$. Moreover, a $\ZN$ $\phi$-ring is said to be a \emph{strong $\phi$-ring}.  Many well-known notions of integral domains have the corresponding analogues in  the class of $\phi$-rings, such as valuation domains, \Dedekind\ domains,  \Prufer\ domains,  Noetherian domains, coherent domains, Bezout domains and Krull domains (see \cite{FA04,FA05,aa16,A03,ALT06}).

The studies of $\phi$-rings from the moduletic  viewpoint started from  Yang \cite{Y06}, who introduced the notion of nonnil-injective modules by replacing the ideals in Baer's criterion for injective modules with nonnil ideals. Dually, Zhao et al. \cite{ZWT13} defined the $\phi$-flat modules in terms of nonnil ideals and $\Tor$-functors. They also gave the conceptions of \emph{$\phi$-von Neumann rings}  over which any module is $\phi$-flat, and then showed that a $\phi$-ring $R$ is  $\phi$-von Neumann if and only if its \Krull\ dimension is $0$,  if and only if $R/\Nil(R)$ is a von Neumann regular ring. In 2018, Zhao \cite{Z18} gave a homological characterization of $\phi$-\Prufer\ rings: a strong $\phi$-ring $R$ is $\phi$-\Prufer\ if and only if each submodule of a $\phi$-flat module is $\phi$-flat, if and only if each nonnil ideal of $R$ is $\phi$-flat. Recently, the first author and Qi \cite{ZQ22-p-dp} characterized $\phi$-von Neumann rings and $\phi$-Dedekind rings in terms of nonnil-injective modules.

It is well-known that the class of flat modules (resp., injective) modules resolving (resp., coresolving). These properties of a given class of  $R$-modules are very crucial to study the homology dimensions (see \cite{EJ00}). So it is natural and worth to ask that: Is the class of $\phi$-flat (resp., nonnil-injective) modules also resolving (resp., coresolving)? The original motivation of this paper is to investigate this question. Actually, we deny these for both  $\phi$-flat modules and nonnil-injective modules  (see Example \ref{flat-nots}, Example \ref{inj-nots}). So we introduce the notions of  strongly $\phi$-flat modules and strongly nonnil-injective modules to fill this gap (see Definition \ref{nonnil-regu}).  The new notions and the old ones are consistent over $\ZN$-ring (see Theorem \ref{sp-ss}). It was proved in \cite{ZZW21} that a $\phi$-ring $R$ is an integral domain if and only if every $\phi$-flat module is flat. However, it does not hold for strongly $\phi$-flat modules  (see Example \ref{flat-nots-5}).
We introduce the $\phi$-flat dimensions and $\phi$-injective dimensions of $R$-modules, and investigate the $\phi$-weak global dimensions and  $\phi$-global dimensions  of rings. It is worth to note that we characterize $\phi$-rings with $\phi$-weak global dimensions and $\phi$-global dimensions  at most $1$.

\section{strongly $\phi$-flat modules and strongly nonnil-injective modules }

Let $R$ be an $\NP$-ring.  Then the set of all nonnil ideals is closed under multiplication. So we always suppose $R$ is an $\NP$-ring from now on. Let $M$ be an $R$-module.  Set
\begin{center}
$\phi$-$tor(M)=\{x\in M|Ix=0$ for some  $I\in \NN(R)\}$.
\end{center}
An $R$-module $M$ is said to be \emph{$\phi$-torsion} (resp., \emph{$\phi$-torsion free}) provided that  $\phi$-$tor(M)=M$ (resp., $\phi$-$tor(M)=0$). Then the classes of $\phi$-torsion modules and $\phi$-torsion free modules constitute a hereditary torsion theory of finite type.

Recall from \cite{ZWT13,ZZ19} that an $R$-module $M$ is called \emph{$\phi$-flat} if $\Tor^R_1(T,M)=0$ for any $\phi$-torsion module $T$; and $M$ is called \emph{nonnil-injective} if $\Ext_R^1(T,M)=0$ for any $\phi$-torsion module $T$. It is well-known that and $R$-module $M$ is  $\phi$-flat  if and only if $\Tor^R_1(R/I,M)=0$  for any (finitely generated) nonnil ideal $I$ of $R$; and $M$ is  nonnil-injective if and only if $\Ext_R^1(R/I,M)=0$ for any nonnil ideal $I$ of $R$.

It is well-known that the class of flat modules is resolving, that is, it contains all projective modules and is closed under direct summands, extensions and kernels of surjective homomorphisms; and  the class of  injective modules is coresolving, that is, it contains all injective modules and is closed under direct summands, extensions and cokernels of injective homomorphisms. And so it is ubiquitous to study modules and rings by using flats and injectives.  So it is natural and worth to ask that:

 \textbf{Is the class of all $\phi$-flat (resp., nonnil-injective) modules resolving (resp., coresolving)?}

To give a negative answer of this question, we first recall the ideal construction given in \cite{DW09}. Let $R$ be a ring and $M$ an $R$-module. Set $R(+)M$ be an $R$-module isomorphic to $R\oplus M$, and define
\begin{enumerate}
    \item ($r,m$)+($s,n$)=($r+s,m+n$),
    \item  ($r,m$)($s,n$)=($rs,sm+rn$).
\end{enumerate}
Then $R(+)M$ become a commutative ring with identity $(1,0)$.

Now, we are ready to give the example to show if $0\rightarrow A\rightarrow B\rightarrow C\rightarrow  0$ is an exact sequence with $B$ and $C$ $\phi$-flat, then $A$ is not necessary $\phi$-flat. Specially, the class of all $\phi$-flat  modules is not resolving.

\begin{example}\label{flat-nots}
Let $\mathbb{Z}$ be the ring of all integers with $\mathbb{Q}$ its quotients field, and  $\mathbb{Z}(\p^{\infty}):=\{\frac{n}{\p^k}+\mathbb{Z}\mid \frac{n}{\p^k}+\mathbb{Z} \in \mathbb{Q}/\mathbb{Z}\}$  the $\p$-Pr\"{u}fer group with $\p$ a prime in $\mathbb{Z}$.
Set $R=\mathbb{Z}(+)\mathbb{Z}(\p^{\infty})$ the trivial extension of $\mathbb{Z}$ with $\mathbb{Z}(\p^{\infty})$. Since $\mathbb{Z}(\p^{\infty})$ is a divisible, we have $R$ is a $\phi$-ring by \cite[Corollary 2.4]{KKM21} where $\Nil(R)= 0(+)\mathbb{Z}(\p^{\infty})$, and so $R/\Nil(R)$ is $\phi$-flat since $\Tor_1^R(R/I,R/\Nil(R))=(I\cap\Nil(R))/I\Nil(R)=0$ for any nonnil ideal $I$ of $R$. However, we claim that $\Nil(R)$ is not  $\phi$-flat. Indeed, let $I=\langle (\p,0)\rangle$. Then $I$ is nonnil. The claim follows by the following isomorphisms (see \cite[Proposition 1]{H60}):
 \begin{align*}
 & \Tor^R_1(R/I,\Nil(R)) \\
 \cong &\{(0,m)\in  0(+)\mathbb{Z}(\p^{\infty})\mid (\p,0)(0,m)=0\}/(0:_R(\p,0))\cdot 0(+)\mathbb{Z}(\p^{\infty}) \\
  \cong &0(+)\mathbb{Z}(\p^1)/0(+)\mathbb{Z}(\p^1)\cdot 0(+)\mathbb{Z}(\p^{\infty}) \\
   \cong&0(+)\mathbb{Z}(\p^1)\not=0,
\end{align*}
where $\mathbb{Z}(\p^1):=\{\frac{n}{p}+\mathbb{Z}\in  \mathbb{Z}(\p^{\infty})\mid n\ \mbox{is an integer}\}$ is a subgroup of $\mathbb{Z}(\p^{\infty})$.
\end{example}

The following example  also shows that if $0\rightarrow A\rightarrow B\rightarrow C\rightarrow  0$ is an exact sequence with $A$ and $B$ nonnil-injective, then $C$ is not necessary nonnil-injective. Specially, the class of all nonnil-injective  modules is not coresolving.
\begin{example}\label{inj-nots} Consider the above Example \ref{flat-nots}. Let $E:=\Hom_{\mathbb{Z}}(R/\Nil(R),\mathbb{Q}/\mathbb{Z})$. Then $E$ is nonnil-injective by \cite[Proposition 1.4]{ZQ22-p-dp}. However, we claim the quotient $\Hom_{\mathbb{Z}}(\Nil(R),\mathbb{Q}/\mathbb{Z})$ of  the injective module $\Hom_{\mathbb{Z}}(R,\mathbb{Q}/\mathbb{Z})$ by $E$  is not nonnil-injective. Indeed,
$$\Ext_R^1(R/I,\Hom_{\mathbb{Z}}(\Nil(R),\mathbb{Q}/\mathbb{Z}))
 \cong \Hom_{\mathbb{Z}}(\Tor^R_1(R/I,\Nil(R)),\mathbb{Q}/\mathbb{Z})\not=0$$
by Example \ref{flat-nots}. Hence,  $\Hom_{\mathbb{Z}}(\Nil(R),\mathbb{Q}/\mathbb{Z})$ is not nonnil-injective.
\end{example}

In view of the above examples, the class of all $\phi$-flat  modules is not resolving, and the class of all nonnil-injective is not coresolving in general.
To obtain the resolving or coresolving property similar to flatness and injectivity in $\NP$-rings, we introduce the following ``strong version'' of $\phi$-flat modules and  nonnil-injective modules using higher derived functors.

\begin{definition}\label{nonnil-regu}
Let $R$ be an $\NP$-ring and $M$ an $R$-module. Then
 \begin{enumerate}
    \item $M$ is called \emph{strongly $\phi$-flat} if $\Tor^R_n(T,M)=0$ for any $\phi$-torsion module $T$ and any $n\geq 1$.
        \item $M$ is called \emph{strongly nonnil-injective} if $\Ext_R^n(T,M)=0$ for any $\phi$-torsion module $T$ and any $n\geq 1$.
 \end{enumerate}
\end{definition}

\begin{lemma}\label{nonnil-regu}
Let $R$ be a $\phi$-ring and $M$ an $R$-module. Then
 \begin{enumerate}
    \item $M$ is strongly $\phi$-flat if and only if  $\Tor^R_n(R/I,M)=0$ for any $($finitely generated$)$ nonnil ideal $I$ of $R$ and any $n\geq 1$.
        \item $M$ is strongly nonnil-injective if and only if\ $\Ext_R^n(R/I,M)=0$ for any nonnil ideal $I$ of $R$ and any $n\geq 1$.
 \end{enumerate}
\end{lemma}
\begin{proof} One can easily verify that  an $R$-module $M$ is strongly $\phi$-flat (resp.,  strongly nonnil-injective) if and only if each syzygies $\Omega^{n}(M)$ (resp., co-syzygies $\Omega^{-n}(M)$) of $M$ is $\phi$-flat  (resp., nonnil-injective) and that  each $\Omega^{n}(M)$ (resp., $\Omega^{-n}(M)$) is  $\phi$-flat (resp., nonnil-injective) if and only if $\Tor^R_1(R/I,\Omega^{n}(M))=0$ for any nonnil ideal $I$ of $R$. (resp., $\Ext_R^1(R/I,\Omega^{-n}(M))=0$ for any $($finitely generated$)$  nonnil ideal $I$ of $R$.)
\end{proof}

 \begin{proposition}\label{bas-sfsi}
Let $R$ be a $\phi$-ring and  $0\rightarrow A\rightarrow B\rightarrow C\rightarrow 0$ a short exact sequence of $R$-modules.  Then the following statements hold.
 \begin{enumerate}
    \item The class of strongly $\phi$-flat modules $($resp.,  strongly nonnil-injective modules$)$ is closed under direct limits $($resp., direct products$),$ direct summands, and extensions.
    \item If $B$ and $C$ are  strongly $\phi$-flat modules, so is $A$.
    \item If  $A$ and $B$ are  strongly nonnil-injective modules, so is $C$.
 \end{enumerate}
\end{proposition}
\begin{proof} We only prove $(2)$, since the proof of  $(1)$ is easy and the proof of  $(3)$ is similar to that of $(2)$. Let $T$ be a $\phi$-torsion module. Then we have an exact sequence  $\cdots\rightarrow \Tor_{n+1}^R(T,C)\rightarrow \Tor_n^R(T,A)\rightarrow \Tor_n^R(T,B)\rightarrow \cdots \rightarrow \Tor_2^R(T,C)\rightarrow \Tor_1^R(T,A)\rightarrow \Tor_1^R(T,B)\rightarrow \Tor_1^R(T,C).$ Since  $B$ and $C$ are  strongly $\phi$-flat modules, $\Tor^R_n(T,B)=\Tor^R_n(T,C)=0$ for any $n\geq 1$. Hence $\Tor^R_n(T,A)=0$ for any $n\geq 1$, whence $A$ is  strongly $\phi$-flat.
\end{proof}

Obviously, every strongly $\phi$-flat module is $\phi$-flat, and every strongly nonnil-injective module is nonnil-injective. However, it follows by Lemma \ref{bas-sfsi} that Example \ref{flat-nots} and  Example \ref{inj-nots} show that $\phi$-flat modules are not always strongly $\phi$-flat, and nonnil-injective modules are also not always strongly nonnil-injective.  But the following result exhibits that over \ZN\ rings, $\phi$-flat modules are exactly strongly $\phi$-flat and nonnil-injective modules are exactly  strongly nonnil-injective.
\begin{theorem}\label{sp-ss}
Let $R$ be  a \ZN\ ring.  Then the following  assertions hold:
\begin{enumerate}
    \item  an $R$-module $M$ is  $\phi$-flat if and only if it is strongly $\phi$-flat;
    \item  an $R$-module $M$ is  nonnil-injective if and only if it is strongly nonnil-injective.
\end{enumerate}
\end{theorem}
\begin{proof} (1) Suppose $M$ is a $\phi$-flat $R$-module. Let $J$ be an ideal  of $R$ that contains $a$. Since $a$ is a non-zero-divisor of $R$, $\Tor_n^R(R/\langle a\rangle,M)=0$ for any positive integer $n$.  It follows by \cite[Proposition 4.1.1]{CE56} that  $$\Tor_1^{R/\langle a\rangle}(R/J,M/aM)\cong \Tor_1^{R/\langle a\rangle}(R/J,M\otimes_RR/\langle a\rangle)\cong \Tor_1^{R}(R/J,M)=0.$$ Hence $M/Ma$ is a flat $R/\langle a\rangle$-module. Consequently, for any $n\geq 1$ we have $$\Tor_n^R(R/J,M)\cong \Tor_n^{R/\langle a\rangle}(R/J,M\otimes_RR/\langle a\rangle)\cong \Tor_n^{R/\langle a\rangle}(R/J,M/aM)=0.$$ It follows that $M$ is a strongly $\phi$-flat $R$-module.

(2) Now suppose $M$ is a nonnil-injective $R$-module. Let $J$ be an ideal  of $R$ that contains $a$.  Since $a$ is a non-zero-divisor of $R$, $\Ext^n_R(R/\langle a\rangle,M)=0$ for any positive integer $n$.  It follows by \cite[Proposition 4.1.4]{CE56}  that  $$\Ext^1_{R/\langle a\rangle}(R/J,\Hom_R(R/\langle a\rangle,M))\cong \Ext^1_{R}(R/J,M)=0.$$ Hence $\Hom_R(R/\langle a\rangle,M)$ is an injective $R/\langle a\rangle$-module by Baer criterion. Consequently, for any $n\geq 1$ we have $$\Ext^n_R(R/J,M)\cong \Ext^n_{R/\langle a\rangle}(R/J,\Hom_R(R/\langle a\rangle,M))=0.$$ It follows that $M$ is a strongly nonnil-injective $R$-module.
\end{proof}

\begin{remark}  Recall from \cite{QZ22-regular-Noe,XWL22} that an $R$-module $M$ is called to be regular flat (resp., regular injective) if  $\Tor^R_1(R/I,M)=0$ (resp., $\Ext_R^1(R/I,M)=0$) for any regular ideal (i.e., an ideal that contains a non-zero-divisor) $I$ of $R$. Similar with the proof of Theorem \ref{sp-ss}, one can show that an  $R$-module $M$ is  regular flat (resp., regular injective) if and only if  $\Tor^R_n(R/I,M)=0$ (resp., $\Ext_R^n(R/I,M)=0$) for any regular ideal $I$ of $R$ and any $n\geq 1$.
\end{remark}

It is known that \ZN\ $\phi$-ring is exactly a strongly $\phi$-ring. The following result is devoted to  the converse of Theorem \ref{sp-ss} under some assumptions.

\begin{theorem}\label{op-fn}
Let $R$ be  a $\phi$-ring with $(0:_Ra)$ finitely generated for any non-nilpotent element $a$ $($e.g. $R$ is a nonnil-coherent ring$)$ or $\Nil(R)$ nilpotent.  If one of the following two statements holds
\begin{enumerate}
    \item every  $\phi$-flat  $R$-module is  strongly $\phi$-flat;
\item every nonnil-injective  $R$-module is  strongly nonnil-injective,
\end{enumerate}
then $R$ is a strongly $\phi$-ring.
\end{theorem}
\begin{proof}
(1) Let $R$ be  a $\phi$-ring and  $a$ a non-nilpotent element in $R$. Suppose  every  $\phi$-flat  $R$-module is  strongly $\phi$-flat. It follows by the proof of \cite[Proposition 1]{ZZW21} that $R/\Nil(R)$ is a $\phi$-flat  $R$-module, and so is strongly $\phi$-flat. Hence,  $$\Tor_2^R(R/Ra,R/\Nil(R))\cong\Tor_1^R(R/(0:_Ra),R/\Nil(R))\cong \frac{(0:_Ra)\cap\Nil(R)}{(0:_Ra)\Nil(R)}=0.$$ Since $R$ is  a $\phi$-ring,    $(0:_Ra)\subseteq \Nil(R)$, and so $(0:_Ra)\cap\Nil(R)=(0:_Ra)$. So  $\Tor_2^R(R/Ra,R/\Nil(R))\cong \frac{(0:_Ra)}{(0:_Ra)\Nil(R)}=0$. And hence $(0:_Ra)=(0:_Ra)\Nil(R)$.

(a) Suppose $(0:_Ra)$ is finitely generated. By Nakayama's lemma, we have $(0:_Ra)=0$, that is, $a$ is a nonzero-divisor. So $R$ is a strongly $\phi$-ring.

(b)  Suppose $\Nil(R)$ is nilpotent.  Assume $\Nil(R)^m=0$. Then $(0:_Ra)=(0:_Ra)\Nil(R)=\cdots=(0:_Ra)\Nil(R)^m=0$. So $R$ is a strongly $\phi$-ring.

(2)  Let $R$ be  a $\phi$-ring and  $a$ a non-nilpotent element in $R$. Suppose every nonnil-injective  $R$-module is  strongly nonnil-injective. It follows by the proof of  \cite[Theorem 1.6]{ZQ22-p-dp} that $(R/\Nil(R))^+:=\Hom_{\mathbb{Z}}(R/\Nil(R),\mathbb{Q}/\mathbb{Z})$ is a nonnil-injective  $R$-module, and so is strongly nonnil-injective, where $E$ is an injective cogenerator of $R$-modules.  Hence,  $$\Ext^2_R(R/Ra,(R/\Nil(R))^+)\cong\Tor_2^R(R/Ra,R/\Nil(R))^+=0.$$
So $\Tor_2^R(R/Ra,R/\Nil(R))=0$, and hence $(0:_Ra)=(0:_Ra)\Nil(R)$. The rest is the same  with that of $(1)$.
\end{proof}




\begin{proposition}\label{flat-FP-injective}
Let $R$ be an $\NP$-ring, then the following assertions are equivalent:
\begin{enumerate}
    \item $M$ is strongly $\phi$-flat;
      \item $\Hom_R(M,E)$ is strongly nonnil-injective for any injective module $E$;
    \item  if $E$ is an injective cogenerator, then $\Hom_R(M,E)$ is strongly nonnil-injective.
\end{enumerate}
\end{proposition}
\begin{proof}
$(1)\Rightarrow (2)$: Let $T$ be a  $\phi$-torsion $R$-module and $E$ an injective $R$-module. Since $M$ is strongly $\phi$-flat, $$\Ext_R^n(T,\Hom_R(M,E))\cong\Hom_R(\Tor_n^R(T,M),E)=0$$ for any positive integer $n$.  Thus $\Hom_R(M,E)$ is strongly  nonnil-injective.

$(2)\Rightarrow (3)$: Trivial.

$(3)\Rightarrow (1)$:  Let $I$ be a  nonnil ideal of $R$ and $E$ an injective cogenerator. Since $\Hom_R(M,E)$ is strongly nonnil-injective,  $$\Hom_R(\Tor_n^R(R/I,M),E)\cong \Ext_R^n(R/I,\Hom_R(M,E))=0$$ for any positive integer $n$. Since $E$ is an injective cogenerator, $\Tor_n^R(R/I,M)=0$ for any positive integer $n$. Thus $M$ is strongly $\phi$-flat by Lemma \ref{nonnil-regu}.
\end{proof}

Let $R$ be an $\NP$-ring. Then every flat $R$-module is strongly $\phi$-flat, and every injective $R$-module is strongly nonnil-injective.  The converses are trivially true for integral domains, but not in general.

\begin{example}  It is obvious that all flat (resp., injective) modules are  strongly $\phi$-flat (resp., strongly nonnil-injective). However, the converse does not hold in general. Indeed, let $R$ be a strongly $\phi$-ring which is not an integral domain (e.g. $R=D(+)Q$ with $D$ a domain and $Q$ its quotient field). Then every  strongly $\phi$-flat (resp., strongly nonnil-injective) module is $\phi$-flat (resp., nonnil-injective) by Theorem \ref{sp-ss}. However, there exist $\phi$-flat (resp., nonnil-injective) modules which are not flat (resp., injective), (see \cite[Proposition 1]{ZZW21} and \cite[Theorem 1.6]{ZQ22-p-dp}).
\end{example}

It was proved in \cite[Proposition 1]{ZZW21} and \cite[Theorem 1.6]{ZQ22-p-dp}  that a $\phi$-ring $R$ is an integral domain if and only if every $\phi$-flat  $R$-module is flat, if and only if  every nonnil-injective  $R$-module is injective. The following example shows that all strongly $\phi$-flat (resp., strongly nonnil-injective $\phi$-torision-free) modules can be flat (resp., injective) over $\phi$-rings which are not domains.

\begin{example}\label{flat-nots-5}
Let $R=\mathbb{Z}(+)\mathbb{Z}(\p^{\infty})$ be the ring in Example \ref{flat-nots}. Then the following statements hold.
\begin{enumerate}
    \item Every strongly $\phi$-flat $R$-module is flat.
     \item Every strongly nonnil-injective $\phi$-torision-free $R$-module is injective.
\end{enumerate}
\end{example}
\begin{proof}
Let $I$ be an ideal of $R$. Then, by \cite[Corollary 3.4]{DW09}, $I$ is of the following two form:
 \begin{enumerate}[(a)]
    \item $I:=\langle (n,0)\rangle=\langle n\rangle (+)\mathbb{Z}(\p^{\infty})$ where $0\not=n\in \mathbb{Z}$;
        \item $I:=0(+)N$, where $N$ is a subgroup of $\mathbb{Z}(\p^{\infty})$.
 \end{enumerate}
The ideal $I$ in case (a) is a nonnil ideal of $R$. Now we consider the ideal in case (b). Then $N$ is of the form $\mathbb{Z}(\p^k):=\{\frac{n}{\p^k}+\mathbb{Z}\in  \mathbb{Z}(\p^{\infty})\mid n\ \mbox{is an integer}\}$ with $k$ a non-negative integer or $\mathbb{Z}(\p^{\infty})$. Set $I_k:=\langle (0,\frac{1}{\p^k})\rangle=0(+)\mathbb{Z}(\p^k)$ for each positive integer $k$.
Note that there is a short exact sequence
$0\rightarrow I_k\rightarrow R\rightarrow(\p^k,0)R\rightarrow 0$
for each non-negative integer $k$. Note that $\mathbb{Z}(\p^{\infty})=\bigcup\mathbb{Z}(\p^k)=\lim\limits_{\longrightarrow}\mathbb{Z}(\p^k)$. Set $I_{\infty}=0(+)\mathbb{Z}(\p^\infty)$, then $I_{\infty}=\lim\limits_{\longrightarrow}I_k$.

(1) Suppose $M$ is a strongly $\phi$-flat $R$-module.
It follows that
$$\Tor^R_1(R/I_k,M)\cong \Tor^R_1(\langle(\p^k,0)\rangle,M)\cong \Tor^R_2(R/\langle(\p^k,0)\rangle,M)=0$$
for each positive integer $k$.
And so each natural homomorphism $f_k:I_k\otimes_RM\rightarrow R\otimes_RM$ is a monomorphism.
Now consider the case $I_{\infty}$. Then the natural map $f_\infty:I_{\infty}\otimes_RM\rightarrow R\otimes_RM$, which can be seen as the direct limits of $f_k$, is also a monomorphism. So $\Tor^R_1(R/I_{\infty},M)=0$. In conclusion,  $\Tor^R_1(R/I,M)=0$ for any ideal $I$ of $R$. It follows that  $M$ is a  flat $R$-module.

(2) Suppose $M$ is a strongly nonnil-injective $\phi$-torision-free  $R$-module. Then $$\Ext_R^1(R/I_k,M)\cong \Ext_R^1((\p^k,0)R,M)\cong \Ext_R^2(R/(\p^k,0)R,M)=0$$
for each non-negative integer $k$. Now, consider the  the case $I_{\infty}$.
 Let $$0\rightarrow \Hom_R(R/I_k,M)\rightarrow \Hom_R(R,M)\rightarrow \Hom_R(I_k,M)\rightarrow 0$$ be the natural exact sequence. Taking inverse limits, we have the following exact sequence: $$0\rightarrow \lim\limits_{\longleftarrow}\Hom_R(R/I_k,M)\rightarrow \lim\limits_{\longleftarrow}\Hom_R(R,M)\rightarrow \lim\limits_{\longleftarrow}\Hom_R(I_k,M)\rightarrow {\textstyle\lim\limits_{\longleftarrow}}^1\Hom_R(R/I_k,M) \rightarrow 0$$
  by \cite[1.2.2]{SS18}.
Considering the $R$-exact sequence $0\rightarrow I_{k+1}/I_k\rightarrow R/I_{k}\rightarrow R/I_{k+1}\rightarrow 0,$  we have an exact sequence $$0\rightarrow\Hom_R(R/I_{k+1},M)\rightarrow \Hom_R(R/I_{k},M)\rightarrow \Hom_R(I_{k+1}/I_k,M)\rightarrow0.$$  Since $(0:_RI_{k+1}/I_k)=(0:_RI_1)=\langle (p,0)\rangle$, we have $$\Hom_R(I_{k+1}/I_k,M)\cong \Hom_R(R/\langle (p,0)\rangle,M)=0$$ as $M$ is $\phi$-torsion-free. So we have a natural isomorphism $\Hom_R(R/I_{k+1},M)\cong \Hom_R(R/I_{k},M)$ for each non-negative integer $k$, and hence the inverse system $\{\Hom_R(R/I_{k},M)\mid k\geq 0\}$ is Mittag-Leffler.  It follows by \cite[1.2.3]{SS18} that $${\textstyle\lim\limits_{\longleftarrow}}^1\Hom_R(R/I_k,M)=0.$$ Consequently, the natural map $$\lim\limits_{\longleftarrow}\Hom_R(R,M)\cong\Hom_R(R,M)\twoheadrightarrow \lim\limits_{\longleftarrow}\Hom_R(I_k,M)\cong\Hom_R(I_\infty,M)$$ is an epimorphism and so $\Ext_R^1(R/I_\infty,M)=0$. In conclusion, $\Ext_R^1(R/I,M)=0$ for any ideal $I$ of $R$. It follows that $M$ is an injective $R$-module.
\end{proof}

\section{on $\phi$-flat dimensions of modules and $\phi$-weak global dimensions of rings}

Let $R$ be a ring. It is well known that the flat dimension of an $R$-module $M$ is defined as the shortest flat resolution of $M$ and the weak global dimension of $R$ is the supremum of the flat dimensions  of all $R$-modules.  We now introduce the notion of $\phi$-flat dimension of an $R$-module as follows.

\begin{definition}\label{w-phi-flat }
Let $R$ be a ring and $M$ an $R$-module. We write $\phi$-\fd$_R(M)\leq n$  ($\phi$-\fd\ abbreviates  \emph{$\phi$-flat dimension}) if there is an exact sequence of $R$-modules
$$ 0 \rightarrow F_n \rightarrow ...\rightarrow F_1\rightarrow F_0\rightarrow M\rightarrow 0   \ \ \ \ \ \ \ \ \ \ \ \ \ \ \ \ \ \ \ \ \ \ \ \ \ \ \ \ \ \ \ \ \ \ \ \ \ \ \ (\diamondsuit)$$
where each $F_i$ is  strongly $\phi$-flat for $i=0,...,n$. The exact sequence $(\diamondsuit)$ is said to be  a $\phi$-flat resolution of length $n$ of $M$. If such finite resolution does not exist, then we say $\phi$-\fd$_R(M)=\infty$; otherwise,  define $\phi$-\fd$_R(M) = n$ if $n$ is the length of the shortest $\phi$-flat resolution of $M$.
\end{definition}\label{def-wML}
It is obvious that an $R$-module $M$ is strongly $\phi$-flat if and only if $\phi$-\fd$_R(M)=0$. Certainly, $\phi$-\fd$_R(M)\leq$\fd$_R(M)$. If $R$ is an integral domain, then $\phi$-\fd$_R(M)=$\fd$_R(M).$

\begin{proposition}\label{w-phi-flat d}
Let $R$ be an $\NP$-ring. The following statements are equivalent for an $R$-module $M$:
\begin{enumerate}
    \item $\phi$-\fd$_R(M)\leq n$;
    \item $\Tor^R_{n+k}(T, M)=0$ for all $\phi$-torsion $R$-modules $T$ and all positive integer $k$;
        \item $\Tor^R_{n+k}(R/I,M)=0$ for all nonnil ideals $I$ and all positive integer $k$;
    \item $\Tor^R_{n+k}(R/I,M)=0$ for all finitely generated nonnil ideals $I$ and all positive integer $k$;
    \item if $0 \rightarrow F_n \rightarrow ...\rightarrow F_1\rightarrow F_0\rightarrow M\rightarrow 0$ is an exact sequence, where $F_0, F_1, . . . , F_{n-1}$ are strongly $\phi$-flat $R$-modules, then $F_n$ is strongly $\phi$-flat;
    \item if $0 \rightarrow F_n \rightarrow ...\rightarrow F_1\rightarrow F_0\rightarrow M\rightarrow 0$ is an exact sequence, where $F_0, F_1, . . . , F_{n-1}$ are flat $R$-modules, then $F_n$ is strongly $\phi$-flat;
    \item there exists an exact sequence  $0 \rightarrow F_n \rightarrow ...\rightarrow F_1\rightarrow F_0\rightarrow M\rightarrow 0$ ,
where $F_0, F_1, . . . , F_{n-1}$ are flat $R$-modules, then $F_n$ is strongly  $\phi$-flat.
\end{enumerate}
\end{proposition}
\begin{proof}
$(1) \Rightarrow(2)$: We prove $(2)$ by induction on $n$.  For the case $n = 0$, (2) trivially holds  as $M$ is strongly $\phi$-flat. If $n>0$, then
there is an exact sequence  $0 \rightarrow F_n \rightarrow ...\rightarrow F_1\rightarrow F_0\rightarrow M\rightarrow 0$,
where each $F_i$ is strongly  $\phi$-flat for $i=0,...,n$. Set $K_0 = \ker(F_0\rightarrow M)$. Then both
$0 \rightarrow  K_0 \rightarrow  F_0 \rightarrow  M \rightarrow  0 $ and $0 \rightarrow  F_n \rightarrow  F_{n-1} \rightarrow...\rightarrow  F_1 \rightarrow  K_0 \rightarrow  0$ are exact. So $\phi$-\fd$_R(K_0)\leq n-1$. By induction, $\Tor^R_{n-1+k}(T, K_0)=0$
for all $\phi$-torsion $R$-modules $T$ and all positive integer $k$. Thus, it follows from the exact sequence   $$0=\Tor^R_{n+k}(T,F_0 )\rightarrow\Tor^R_{n+k}(T,M )\rightarrow \Tor^R_{n-1+k}(T, K_0 )\rightarrow \Tor^R_{n-1+k}(T, F_0)=0$$ that
$\Tor^R_{n+k}(T, M)\cong \Tor^R_{n-1+k}(T, K_0 )=0$.

$(2) \Rightarrow(3)\Rightarrow(4)$ and $(5) \Rightarrow(6)$:  Trivial.

$(4) \Rightarrow(5)$: Let $K_0 = \ker(F_0 \rightarrow  M)$ and $K_i = \ker(F_i \rightarrow  F_{i-1})$, where
$i = 1, . . . , n-1$. Then $K_{n-1}\cong F_n$. Since all $F_0, F_1, . . . , F_{n-1}$ are strongly $\phi$-flat,
$\Tor^R_k (R/I, F_n)\cong\Tor^R_{1+k} (R/I, K_{n-2})\cong\cdots\cong \Tor^R_{n+k}(R/I, M)=0$  for all finitely generated nonnil ideal $I$ and any  positive integer $k$   by dimensional shift. Hence $F_n$
is strongly  $\phi$-flat by Lemma \ref{nonnil-regu}.

$(6) \Rightarrow (7)$: Since the class of flat modules is covering, we can construct an exact sequence $\cdots \rightarrow F_{n-1} \xrightarrow{d_{n-1}} F_{n-2}\rightarrow  ...\rightarrow F_1\rightarrow F_0\rightarrow M\rightarrow 0$, where $F_0, F_1, . . . , F_{n-1}$ are flat $R$-modules, then $F_n:=\Ker(d_{n-1})$ is strongly  $\phi$-flat by $(6)$.

$(7) \Rightarrow (1)$: Trivial.
\end{proof}

The proof the following two results are similar with the classical ones and so we omit its proof.

\begin{corollary}
Let $R$ be an $\NP$-ring and $0\rightarrow A\rightarrow B\rightarrow C\rightarrow 0$ be an exact sequence of $R$-modules. Then the following results hold.
\begin{enumerate}
    \item $\phi$-\fd$_R(C)\leq 1+\max\{\phi$-\fd$_R(A),\phi$-\fd$_R(B)\}$;
    \item if $\phi$-\fd$_R(B)<\phi$-\fd$_R(C)$, then $\phi$-\fd$_R(A)= \phi$-\fd$_R(C)-1\geq\phi$-\fd$_R(B)$.
\end{enumerate}
\end{corollary}

\begin{corollary}
Let $R$ be an $\NP$-ring and $\{M_i\mid i\in\Gamma\}$ be a direct system of  $R$-modules. Then
\begin{center}
 $\phi$-\fd$_R(\lim\limits_{\longrightarrow}M_i)=\sup\{\phi$-\fd$_R(M_i)\}.$
\end{center}
\end{corollary}

Now, we are ready to introduce the $\phi$-weak global dimension of a ring  in terms of $\phi$-flat dimensions.
\begin{definition}\label{w-phi-flat }
The \emph{$\phi$-weak global dimension} of a ring $R$ is defined by
\begin{center}
$\phi$-\cwd$(R) = \sup\{\phi$-\fd$_R(M) | M $ is an $R$-module$\}.$
\end{center}
\end{definition}\label{def-wML}
Obviously, by definition, $\phi$-\cwd$(R)\leq $\cwd$(R)$.  Notice that if $R$ is an integral domain, then $\phi$-\cwd$(R)=$\cwd$(R)$. The following result can easily deduced by Proposition \ref{w-phi-flat d} and so we omit its proof.

\begin{theorem}\label{w-g-flat}
Let $R$ be an $\NP$-ring. The following statements are equivalent for $R$.
\begin{enumerate}
 \item  $\phi$-\cwd$(R)\leq  n$.
    \item  $\phi$-\fd$_R(M)\leq n$ for all $R$-modules $M$.
    \item $\Tor^R_{n+k}(T, M)=0$ for all $R$-modules $M$, all $\phi$-torsion $T$ and all positive integer $k$.

     \item  $\Tor^R_{n+k}(R/I, M)=0$ for all $R$-modules $M$, all  nonnil ideals $I$ of $R$ and all positive integer $k$.
    \item  $\Tor^R_{n+k}(R/I, M)=0$ for all $R$-modules $M$, all  finitely generated nonnil ideals $I$ of $R$ and all positive integer $k$.
            \item $\Tor^R_{n+1}(T, M)=0$ for all $R$-modules $M$ and all  $\phi$-torsion $T$.

     \item  $\Tor^R_{n+1}(R/I, M)=0$ for all $R$-modules $M$ and all  nonnil ideals $I$ of $R$.
    \item  $\Tor^R_{n+1}(R/I, M)=0$ for all $R$-modules $M$ and all  finitely generated nonnil ideals $I$ of $R$.
    \item \fd$_R(R/I)\leq n$ for all  nonnil ideals $I$ of $R$.
    \item \fd$_R(R/I)\leq  n$ for all finitely generated   nonnil ideals $I$ of $R$.

\end{enumerate}
Consequently, the $\phi$-weak global dimension of $R$ is determined by the
formulas:
\begin{align*}
\phi\mbox{-}\cwd(R)&= \sup \{\mbox{\fd}_R(R/I) |\ I\ is\ a\  nonnil\ ideal\ of\ R\}\\
&= \sup \{\mbox{\fd}_R(R/I) |\ I\ \mbox{is\ a\ finitely\ generated\ nonnil\ ideal\ of}\ R\} .
\end{align*}
\end{theorem}

\begin{theorem}\label{ideal-cwd}
Let $R$ be a  strongly $\phi$-ring. Then the following statements hold.
\begin{enumerate}
 \item $\cwd(R/\Nil(R))\leq \phi\mbox{-}\cwd(R).$

 \item $\phi\mbox{-}\cwd(R)-\fd_R(R/\Nil(R))\leq \cwd(R/\Nil(R)).$
\end{enumerate}
\end{theorem}
\begin{proof} (1) Suppose  $\cwd(R/\Nil(R))= n$.
Then there exists  a  nonnil ideal $I$ of $R$ and an $R/\Nil(R)$-module $M$ such that $$\Tor^{R/\Nil(R)}_{n}(R/I\otimes_RR/\Nil(R),M)\cong\Tor^{R/\Nil(R)}_{n}(R/I,M)\not=0.$$
Note that $R/\Nil(R)$ is $\phi$-flat, and then, by Theorem \ref{sp-ss}, we have $\Tor_n^R(R/I,R/\Nil(R))=0$ for all $n\geq 1.$ So $$\Tor^{R}_{n}(R/I,M)\cong\Tor^{R/\Nil(R)}_{n}(R/I\otimes_RR/\Nil(R),M)\not=0,$$ and hence $\mbox{\fd}_R(R/I)\geq n$. It follows by Theorem \ref{w-g-flat} that $\phi\mbox{-}\cwd(R)\geq n$.

(2) It immediately  follows by \cite[Theorem 3.8.5]{fk16} and Theorem \ref{w-g-flat}.
\end{proof}

It is natural to ask the question:
\begin{question}\label{op2}
Let $R$ be a strongly $\phi$-ring. Is the following equation holds? $$\cwd(R/\Nil(R))=\phi\mbox{-}\cwd(R).$$
\end{question}
We can verify it in the following case.

\begin{proposition}\label{ideal-1}
Let $D$ be an integral domain, $Q$ its quotient field and $V$ a $Q$-linear space. Then $\phi$-w.\gl$(D(+)V)=$w.\gl$(D)$.
\end{proposition}
\begin{proof}  Set $R=D(+)V$. Assume w.\gl$(D)\leq n$. Let $M$ be an $R$-module, Then $M$ is naturally an $D$-module. Let $J$ be an nonnil ideal of $R$. Then by
 \cite[Corollary3.4]{DW09}, we have $J=I(+)V$ with $I$ a nonzero $D$-ideal. Note that  $R$ is a flat $D$-module. By \cite[Proposition 4.1.2]{CE56} we have $$\Tor_{n+1}^R(R/J,M)\cong\Tor_{n+1}^R(D/I\otimes_DR,M)\cong \Tor_{n+1}^D(D/I,M)=0.$$ So $\phi$-w.\gl$(D(+)V)\leq $w.\gl$(D)$. The result follows by Theorem \ref{ideal-cwd}.
\end{proof}

It is well known that a  ring $R$ with weak global dimension $0$ is exactly a \emph{von Neumann regular ring}, equivalently $a\in (a^2)$ for any $a\in R$.  Recall from  \cite{ZWT13} that a $\phi$-ring $R$ is said to be \emph{$\phi$-von Neumann regular} provided that every $R$-module is $\phi$-flat. A $\phi$-ring $R$  is  $\phi$-von Neumann regular, if and only if there is an element $x\in R$ such that $a = xa^2$ for any non-nilpotent element $a\in R$, if and only if $R/\Nil(R)$ is a von Neumann regular $\phi$-ring, i.e.,  $R/\Nil(R)$ is a field (see \cite[Theorem 4.1]{ZWT13}). Now, we characterize $\phi$-von Neumann regular rings in terms of strongly $\phi$-flat modules and $\phi$-weak global dimensions.

\begin{theorem}\label{w-g-flat-0}
Let $R$ be a $\phi$-ring. The following statements are equivalent for $R$:
\begin{enumerate}
    \item $\phi$-\cwd$(R) =0$;
    \item every $R$-module is strongly $\phi$-flat;
    \item  $R$  is a $\phi$-von Neumann regular ring.
\end{enumerate}
\end{theorem}
\begin{proof} $(1)\Leftrightarrow (2)$ By definition.

$(2)\Rightarrow (3)$: It follow by \cite[Theorem 4.1]{ZWT13}.

$(3)\Rightarrow (2)$: Suppose $R$  is a $\phi$-von Neumann regular ring. Then we claim that $R$ is a $\ZN$-ring. Indeed, let $a$ be a non-nilpotent element in $R$. Since $R/\Nil(R)$ is a field by \cite[Theorem 4.1]{ZWT13}, we have $(1-ab)^n=0$ for some $b\in R$ and positive integer $n$. So $a$ is a unit, and thus a non-zero-divisor. Now $(2)$ follows by Theorem \ref{sp-ss} and \cite[Theorem 4.1]{ZWT13}.
\end{proof}

For a  $\phi$-ring $R$, there is a ring homomorphism $\phi:\T(R)\rightarrow R_{\Nil(R)}$ such that $\phi(a/b)=a/b$ where $a\in R$ and $b$ is a regular element. Denote by the ring $\phi(R)$ the image of $\phi$ restricted to $R$. Then $\phi(R)$ is  a strong  $\phi$-ring.  Recall that a regular ideal $I$ of $R$ is called \emph{invertible} if $II^{-1}=R$ where $I^{-1}=\{x\in\T(R)|Ix\subseteq R\}$.
Recall from \cite{FA04} that a nonnil ideal $I$ of a $\phi$-ring $R$ is said to be  \emph{$\phi$-invertible} provided that $\phi(I)$ is an invertible ideal of $\phi(R)$.

Following \cite{FA04}, a $\phi$-ring  $R$ is said to be a \emph{$\phi$-\Prufer\ ring} if every finitely generated nonnil ideal $I$ is $\phi$-invertible, i.e., $\phi(I)\phi(I^{-1})=\phi(R)$. A $\phi$-ring  $R$ is said to be a \emph{$\phi$-chain ring} ($\phi$-CR for short) if for any $a,b\in R-\Nil(R)$, either $a|b$ or $b|a$ in $R$. It follows from \cite[Corollary 2.10]{FA04} that a $\phi$-ring $R$ is $\phi$-\Prufer, if and only if $R_{\fkm}$ is a $\phi$-CR for any maximal ideal $\fkm$ of $R$, if and only if $R/\Nil(R)$ is a \Prufer\ domain, if and only if $\phi(R)$ is \Prufer.  For a strong $\phi$-ring $R$, Zhao \cite[Theorem 4.3]{Z18} showed that $R$ is a $\phi$-\Prufer\ ring if and only if  all $\phi$-torsion free $R$-modules are $\phi$-flat, if and only if  each submodule of a $\phi$-flat $R$-module is $\phi$-flat, if and only if  each nonnil ideal of $R$ is $\phi$-flat.

\begin{theorem}\label{w-g-flat-1}
Let $R$ be a $\phi$-ring. The following statements are equivalent for $R$:
\begin{enumerate}
    \item $\phi$-\cwd$(R)\leq 1$;
    \item every submodule of flat $R$-module is strong $\phi$-flat;
      \item every submodule of  strong $\phi$-flat $R$-module is strong $\phi$-flat;
    \item $R$  is a $\phi$-\Prufer\ strong $\phi$-ring.
\end{enumerate}
\end{theorem}
\begin{proof} $(1)\Leftrightarrow (2)\Leftrightarrow (3)$ By Theorem \ref{w-g-flat}.

$(4)\Rightarrow (2)$: It follow by \cite[Theorem 4.1]{ZWT13}.

$(2)\Rightarrow (4)$: Since every submodule of flat $R$-module is strong $\phi$-flat, every ideal of $R$ is $\phi$-flat.  It follows by \cite[Corollary 2.8]{KMH22} that $R$ is a strong $\phi$-ring. Hence the result follows by \cite[Theorem 4.3]{Z18} and Theorem \ref{sp-ss}.
\end{proof}

Note that when \cwd$(R/\Nil(R))\leq 1$, then Question \ref{op2} holds by  Theorem \ref{w-g-flat-0} and Theorem \ref{w-g-flat-1}.

\begin{corollary}
Let $D$ be an integral domain, $Q$ its quotient field and $V$ a $Q$-linear space. Then $D(+)V$ is a $\phi$-\Prufer\ ring if and only if $D$ is a \Prufer\ domain.
\end{corollary}
\begin{proof} Note that $D(+)V$ is a strong $\phi$-ring. So the result follows by Proposition \ref{ideal-1}
 and Theorem \ref{w-g-flat-0}.
\end{proof}

The following example shows that the $\phi$-weak global dimensions of $\phi$-\Prufer\ rings can be sufficient large, and so the condition ``$R$  is a strong $\phi$-ring'' in Theorem \ref{w-g-flat-0}(4) cannot be removed.
\begin{example}\label{infty example}
Let $R$ be the ring in Example \ref{flat-nots}.  Then $R$ is a  $\phi$-\Prufer\ rings since $R/\Nil(R)\cong \mathbb{Z}$ is a \Prufer\ domain. It is easy to verify that there is a projective resolution of $\langle p\rangle (+)\mathbb{Z}(\p^{\infty})$ is $$\cdots\rightarrow R\xrightarrow{d_4} R\xrightarrow{d_3} R\xrightarrow{d_2} R\xrightarrow{d_1} R\xrightarrow{d_0} \langle p\rangle (+)\mathbb{Z}(\p^{\infty})\rightarrow0,$$
where $d_n$ is a multiplication by $(\p,0)$ when $n$ is even, and  a multiplication by $(0,\frac{1}{\p}+\mathbb{Z})$ when $n$ is odd. Note that the above projective resolution is not split. So the global dimension, and hence the weak global dimension of $R$ is infinite. By Example \ref{flat-nots-5}, every strongly $\phi$-flat $R$-module is flat. Hence the $\phi$-weak global dimension of $R$ is also infinite.
\end{example}

\section{on $\phi$-injective dimensions of modules and $\phi$-global dimensions of rings}

Let $R$ be a ring. It is well known that the injective dimension of an $R$-module $M$ is defined as the shortest injective resolution of $M$ and the weak global dimension of $R$ is the supremum of the injective dimensions  of all $R$-modules.  We now introduce the notion of $\phi$-injective dimension of an $R$-module as follows.

\begin{definition}\label{w-phi-injective }
Let $R$ be a ring and $M$ an $R$-module. We write $\phi$-id$_R(M)\leq n$  ($\phi$-id abbreviates  \emph{$\phi$-injective dimension}) if there is an exact sequence of $R$-modules
$$ 0\rightarrow  M \rightarrow E_0 \rightarrow E_1\rightarrow \cdots\rightarrow  E_n\rightarrow 0   \ \ \ \ \ \ \ \ \ \ \ \ \ \ \ \ \ \ \ \ \ \ \ \ \ \ \ \ \ \ \ \ \ \ \ \ \ \ \ (\heartsuit)$$
where each $E_i$ is  strongly nonnil-injective for $i=0,...,n$. The exact sequence $(\heartsuit)$ is said to be  a $\phi$-injective resolution of length $n$ of $M$. If such finite resolution does not exist, then we say $\phi$-id$_R(M)=\infty$; otherwise,  define $\phi$-id$_R(M) = n$ if $n$ is the length of the shortest $\phi$-injective resolution of $M$.
\end{definition}\label{def-wML}
It is obvious that an $R$-module $M$ is strongly nonnil-injective if and only if $\phi$-id$_R(M)=0$. Certainly, $\phi$-id$_R(M)\leq$id$_R(M)$. If $R$ is an integral domain, then $\phi$-id$_R(M)=$id$_R(M)$

\begin{proposition}\label{w-phi-injective d}
Let $R$ be an $\NP$-ring. The following statements are equivalent for an $R$-module $M$:
\begin{enumerate}
    \item $\phi$-id$_R(M)\leq n$;
    \item $\Ext_R^{n+k}(T, M)=0$ for all $\phi$-torsion $R$-modules $T$ and all positive integer $k$;
        \item $\Ext_R^{n+k}(R/I,M)=0$ for all nonnil ideals $I$ and all positive integer $k$;
    \item if $0\rightarrow  M \rightarrow E_0 \rightarrow E_1\rightarrow \cdots\rightarrow  E_n\rightarrow 0  $ is an exact sequence, where $E_0, E_1, . . . , E_{n-1}$ are strongly nonnil-injective $R$-modules, then $E_n$ is strongly nonnil-injective;
    \item if $0\rightarrow  M \rightarrow E_0 \rightarrow E_1\rightarrow \cdots\rightarrow  E_n\rightarrow 0 $ is an exact sequence, where $E_0, E_1, . . . , E_{n-1}$ are injective $R$-modules, then $E_n$ is strongly nonnil-injective;
    \item there exists an exact sequence  $0\rightarrow  M \rightarrow E_0 \rightarrow E_1\rightarrow \cdots\rightarrow  E_n\rightarrow 0 $ ,
where $E_0, E_1, . . . , E_{n-1}$ are injective $R$-modules, then $E_n$ is strongly nonnil-injective.
\end{enumerate}
\end{proposition}
\begin{proof}
$(1) \Rightarrow(2)$: We prove $(2)$ by induction on $n$.  For the case $n = 0$, (2) trivially holds  as $M$ is strongly nonnil-injective. If $n>0$, then
there is an exact sequence  $0\rightarrow  M \rightarrow E_0 \rightarrow E_1\rightarrow \cdots\rightarrow  E_n\rightarrow 0$,
where each $E_i$ is strongly nonnil-injective for $i=0,...,n$. Set $K_0 = \Coker(E_0\rightarrow M)$. Then both
$0 \rightarrow  M\rightarrow  E_0 \rightarrow  K_0  \rightarrow  0 $ and $0 \rightarrow  K_0 \rightarrow E_1\rightarrow \cdots\rightarrow  E_n\rightarrow 0$ are exact. So $\phi$-id$_R(K_0)\leq n-1$. By induction, $\Ext_R^{n-1+k}(T, K_0)=0$
for all $\phi$-torsion $R$-modules $T$ and all positive integer $k$. Thus, it follows from the exact sequence   $$0=\Ext_R^{n+k-1}(T,E_0 )\rightarrow\Ext_R^{n+k-1}(T,K_0 )\rightarrow \Ext_R^{n+k}(T, M )\rightarrow \Ext_R^{n+k}(T, E_0)=0$$ that
$\Ext_R^{n+k}(M, T)\cong \Ext_R^{n-1+k}(T, K_0 )=0$.

$(2) \Rightarrow(3)$ and $(4) \Rightarrow(5)$:  Trivial.

$(3) \Rightarrow(4)$: Let $K_0 = \Coker( M \rightarrow E_0)$ and $K_i = \Coker(E_{i-1} \rightarrow  E_i)$, where
$i = 1, . . . , n-1$. Then $K_{n-1}\cong E_n$. Since all $E_0, E_1, . . . , E_{n-1}$ are strongly nonnil-injective,
$\Ext_R^k ( R/I,E_n)\cong\Ext_R^{1+k} ( R/I,K_{n-2})\cong\cdots\cong \Ext_R^{n+k}(R/I, M)=0$  for all  nonnil ideal $I$ and any  positive integer $k$   by dimensional shift. Hence $E_n$
is strongly nonnil-injective by Lemma \ref{nonnil-regu}.

$(5) \Rightarrow (6)$: Consider the injective resolution of $M$: $0\rightarrow  M \rightarrow E_0 \rightarrow E_1\rightarrow \cdots\rightarrow E_{n-2}\xrightarrow{d_{n-2}}   E_{n-1}\rightarrow \cdots$,
where $E_0, E_1,\dots, E_{n-1}$
 are injective $R$-modules. Then $E_n:=\Coker(d_{n-2})$ is strongly nonnil-injective by $(5)$.

$(6) \Rightarrow (1)$: Trivial.
\end{proof}
\begin{corollary}
Let $R$ be an $\NP$-ring, $M$ an $R$-module and $E$ an injective cogenerator of $R$-Mod. Then $\phi$-\fd$_R(M)=\phi$-id$_R(\Hom_R(M,E))$
\end{corollary}
\begin{proof}
It follows by Proposition \ref{w-phi-flat d}, Proposition  \ref{w-phi-injective d} and the adjoint isomorphism: $\Ext_R^n(T,\Hom_R(M,E))\cong\Hom_R(\Tor_n^R(T,M),E)=0$.
\end{proof}

The proof the following two results are similar with the classical ones and so we omit its proof.

\begin{corollary}
Let $R$ be an $\NP$-ring and $0\rightarrow A\rightarrow B\rightarrow C\rightarrow 0$ be an exact sequence of $R$-modules. Then the following results hold.
\begin{enumerate}
    \item $\phi$-id$_R(A)\leq 1+\max\{\phi$-id$_R(B),\phi$-id$_R(C)\}$;
    \item if $\phi$-id$_R(B)<\phi$-id$_R(A)$, then $\phi$-id$_R(C)= \phi$-id$_R(A)-1\geq\phi$-id$_R(B)$.
\end{enumerate}
\end{corollary}

\begin{corollary}
Let $R$ be an $\NP$-ring and $\{M_i\mid i\in\Gamma\}$ be a family of  $R$-modules. Then
\begin{center}
 $\phi$-id$_R(\prod\limits_{i\in\Gamma}M_i)=\sup\{\phi$-id$_R(M_i)\}.$
\end{center}
\end{corollary}

Now, we are ready to introduce the $\phi$-global dimension of a ring  in terms of nonnil-injective dimensions.
\begin{definition}\label{w-phi-injective }
The \emph{$\phi$-global dimension} of a ring $R$ is defined by
\begin{center}
$\phi$-\cgd$(R) = \sup\{\phi$-id$_R(M) | M $ is an $R$-module$\}.$
\end{center}
\end{definition}\label{def-wML}
Obviously, by definition, $\phi$-\cgd$(R)\leq $\cgd$(R)$.  Notice that if $R$ is an integral domain, then $\phi$-\cgd$(R)=$\cgd$(R)$. The following result can easily deduced by Proposition \ref{w-phi-injective d} and so we omit its proof.

\begin{theorem}\label{w-g-injective}
Let $R$ be an $\NP$-ring. The following statements are equivalent for $R$.
\begin{enumerate}
 \item  $\phi$-\cgd$(R)\leq  n$.
    \item  $\phi$-id$_R(M)\leq n$ for all $R$-modules $M$.
    \item $\Ext_R^{n+k}(T, M)=0$ for all $R$-modules $M$,  all  $\phi$-torsion $T$ and all positive integer $k$.

     \item  $\Ext_R^{n+k}(R/I, M)=0$ for all $R$-modules $M$,  all  nonnil ideals $I$ of $R$ and all positive integer $k$.
         \item $\Ext_R^{n+1}(T, M)=0$ for all $R$-modules $M$ and all  $\phi$-torsion $T$ .
     \item  $\Ext_R^{n+1}(R/I, M)=0$ for all $R$-modules $M$ and  all nonnil ideals $I$ of $R$ .
\end{enumerate}
Consequently, the $\phi$-global dimension of $R$ is determined by the
formulas:
\begin{align*}
\phi\mbox{-}\cgd(R)&= \sup \{\mbox{\pd}_R(R/I) |\ I\ is\ a\  nonnil\ ideal\ of\ R\} .
\end{align*}
\end{theorem}

It follows by  Theorem \ref{w-g-flat} and Theorem \ref{w-phi-injective } that $\phi\mbox{-}\cgd(R)\leq \phi\mbox{-}\cwd(R)$ for any $\NP$-rings. We say an $R$-module $M$ is super-finitely presented if $M$ has a projective resolution with each term finitely generated. It is well-known that the weak global dimensions and global dimensions coincide over Noetherian rings. For $\phi$-dimensions, we have the following result.
\begin{corollary}
Let $R$ be a $\NP$-ring such that every nonnil $R$-ideal  is super-finitely presented. Then   $\phi\mbox{-}\cgd(R)= \phi\mbox{-}\cwd(R)$.
\end{corollary}
\begin{proof} Let $I$ be a nonnil ideal of $R$.  Since $R/I$ is super-finitely presented, $\mbox{\pd}_R(R/I)=\mbox{\fd}_R(R/I)$ as finitely presented flat modules are projective. Hence  $\phi\mbox{-}\cgd(R)= \phi\mbox{-}\cwd(R)$ by  Theorem \ref{w-g-flat} and Theorem \ref{w-phi-injective }.
\end{proof}

\begin{remark} Recall from \cite{A03} that a $\phi$-ring $R$ is said to be nonnil-Noetherian if every nonnil ideal is finitely generated. Trivially, if every nonnil $R$-ideal  is super-finitely presented, then $R$ is nonnil-Noetherian. However, the converse does not hold in general. Indeed, let $R=\mathbb{Z}(+)\bigoplus\limits_{i=1}^\infty(\mathbb{Q}/\mathbb{Z})$. Then $R$ is nonnil-Noetherian. But there exists a nonnil $R$-ideal which is not finitely presented (see \cite[Remark 1.1]{QW22} or \cite[Example 4.11]{HKM22}). However, we do not have an example in hand to distinguish $\phi\mbox{-}\cgd(R)$ and $ \phi\mbox{-}\cwd(R)$ over a nonnil-Noetherian ring $R$.
\end{remark}

\begin{theorem}\label{ideal-cgd}
Let $R$ be an $\NP$-ring.  Then the following statements hold.
\begin{enumerate}
 \item If $R$ is a  strongly $\phi$-ring, then $\cgd(R/\Nil(R))\leq \phi\mbox{-}\cgd(R).$

 \item $\phi\mbox{-}\cgd(R)-\fd_R(R/\Nil(R))\leq \cgd(R/\Nil(R)).$
\end{enumerate}
\end{theorem}
\begin{proof} Suppose  $\cgd(R/\Nil(R))= n$.
So there exists  a  nonnil ideal $I$ of $R$ and an $R/\Nil(R)$-module $M$ such that $$\Ext_{R/\Nil(R)}^{n}(R/I,M)\cong\Ext_{R/\Nil(R)}^{n}(R/I\otimes_RR/\Nil(R),M)\not=0.$$
Note that $\Tor_n^R(R/I,R/\Nil(R))=0$ for all $n\geq 1.$ So $$\Ext_{R}^{n}(R/I,M)\cong\Ext_{R/\Nil(R)}^{n}(R/I\otimes_RR/\Nil(R),M)\not=0,$$ and hence $\mbox{\pd}_R(R/I)\geq n$. It follows by Theorem \ref{w-g-injective} that $\phi\mbox{-}\cgd(R)\geq n$.

(2) It immediately follows by \cite[Theorem 3.8.1]{fk16} and Theorem \ref{w-g-injective}.
\end{proof}

It is natural to ask the question:
\begin{question}\label{op1}
Let $R$ be a strongly $\phi$-ring. Is the following equation holds? $$\cgd(R/\Nil(R))=\phi\mbox{-}\cgd(R).$$
\end{question}
We can verify it in the following case.

\begin{proposition}\label{ideal-2}
Let $D$ be an integral domain, $Q$ its quotient field and $V$ a $Q$-linear space. Then $\phi$-\gl$(D(+)V)=$\gl$(D)$.
\end{proposition}
\begin{proof}  Set $R=D(+)V$. Assume \gl$(D)\leq n$. Let $M$ be an $R$-module, Then $M$ is naturally an $D$-module. Let $J$ be an nonnil ideal of $R$. Then by
 \cite[Corollary 3.4]{DW09}, we have $J=I(+)V$ with $I$ a nonzero$D$-ideal. Note that  $R$ is a flat $D$-module. By \cite[Proposition 4.1.3]{CE56} we have $$\Ext^{n+1}_R(R/J,M)\cong\Ext^{n+1}_R(D/I\otimes_DR,M)\cong \Ext^{n+1}_D(D/I,M)=0.$$ So $\phi$-\gl$(D(+)V)\leq $\gl$(D)$. The result follows by Theorem \ref{ideal-cgd}.
\end{proof}

It was proved in \cite[Theorem 1.7]{ZQ22-p-dp} that a $\phi$-ring $R$ is a  $\phi$-von Neumann regular ring if and only if every $R$-module is  nonnil-injective. Moreover, we have the following result.

\begin{theorem}\label{w-g-injective-0}
Let $R$ be a $\phi$-ring. The following statements are equivalent for $R$:
\begin{enumerate}
    \item $\phi$-\cgd$(R)=0$;
    \item every $R$-module is strongly nonnil-injective;
    \item  $R$  is a $\phi$-von Neumann regular ring.
\end{enumerate}
\end{theorem}
\begin{proof} $(1)\Leftrightarrow (2)$ By definition.

$(2)\Rightarrow (3)$: It follow by  \cite[Theorem 1.7]{ZQ22-p-dp}.

$(3)\Rightarrow (2)$: Suppose $R$  is a $\phi$-von Neumann regular ring. Then  $R$ is a $\ZN$-ring by the proof of Theorem \ref{w-g-flat-0}. Now $(2)$ follows by Theorem \ref{sp-ss} and \cite[Theorem 1.7]{ZQ22-p-dp}.
\end{proof}

Recall from  \cite{FA05} that a $\phi$-ring $R$ is called a \emph{$\phi$-Dedekind} ring provided that any nonnil ideal of $R$ is $\phi$-invertible. It was proved in \cite[Theorem 2.5]{FA05} that a $\phi$-ring $R$ is a $\phi$-Dedekind ring if and only if $R/\Nil(R)$ is a Dedekind domain.

\begin{theorem}\label{w-g-injective-1}
Let $R$ be a $\phi$-ring. The following statements are equivalent for $R$:
\begin{enumerate}
    \item $\phi$-\cgd$(R)\leq 1$;
    \item every quotient module of injective $R$-module is strong $\phi$-injective;
      \item every quotient module of  strong $\phi$-injective $R$-module is strong $\phi$-injective;
    \item $R$  is a $\phi$-\Dedekind\ strong $\phi$-ring.
\end{enumerate}
\end{theorem}
\begin{proof} $(1)\Leftrightarrow (2)\Leftrightarrow (3)$ By Theorem \ref{w-g-injective}.

$(4)\Rightarrow (2)$: It follow by  \cite[Theorem 1.7]{ZQ22-p-dp}.

$(2)\Rightarrow (4)$: Suppose every quotient module of injective $R$-module is strong $\phi$-injective. We claim  every ideal of $R$ is strong $\phi$-flat. Indeed, let $I$ be an ideal of $R$. Then for any $\phi$-torsion $R$-module $T$ and positive integer $n$, we have $\Hom_{\mathbb{Z}}(\Tor_n^R(T,I),\mathbb{Q}/\mathbb{Z})\cong \Ext_R^n(T,\Hom_{\mathbb{Z}}(I,\mathbb{Q}/\mathbb{Z}))=0$ since $\Hom_{\mathbb{Z}}(I,\mathbb{Q}/\mathbb{Z})$ is a quotient module of the injective  $R$-module $\Hom_{\mathbb{Z}}(R,\mathbb{Q}/\mathbb{Z})$. Hence $\Tor_n^R(T,I)=0$, whence $I$  is strong $\phi$-flat. It follows by \cite[Corollary 2.8]{KMH22} that $R$ is a strong $\phi$-ring. Hence the result follows by  \cite[Theorem 1.7]{ZQ22-p-dp} and Theorem \ref{sp-ss}.
\end{proof}

\begin{corollary}
Let $D$ be an integral domain, $Q$ its quotient field and $V$ a $Q$-linear space. Then $D(+)V$ is a $\phi$-Dedekind ring if and only if $D$ is a Dedekind domain.
\end{corollary}
\begin{proof} Note that $D(+)V$ is a strong $\phi$-ring. So the result immediately follows by Proposition \ref{ideal-2}
 and Theorem \ref{w-g-injective-0}.
\end{proof}

\begin{remark} When  \cgd$(R/\Nil(R))\leq 1$, then Question \ref{op1} holds by Theorem \ref{w-g-injective-0} and Theorem \ref{w-g-injective-1}. The $\phi$-global dimensions of $\phi$-\Dedekind\ rings can be  large than $1$. Indeed, let $D$ be a \Dedekind\ domain and $Q$ its quotient field. Then $R=D(+)Q/D$ is a $\phi$-\Dedekind\ ring since $R/\Nil(R)\cong \mathbb{Z}$ is a \Dedekind\ domain. However, since $R$ is not a strongly $\phi$-ring, we have $\phi$-\cgd$(R)>1$.
\end{remark}

\begin{acknowledgement}\quad\\
The first author was supported by  the National Natural Science Foundation of China (No. 12061001),
the second author was supported by the Scientific Research Foundation of Chengdu University of Information Technology (No. KYTZ202015,2022ZX001), and the third author was supported by  the National Natural Science Foundation of China (No. 12201361)
\end{acknowledgement}

\end{document}